\documentclass[reqno,12pt]{amsart}
\usepackage{amsmath, latexsym, amsfonts, amssymb, amsthm}
\usepackage{mathrsfs,enumerate}

\newtheorem{theorem}{Theorem}[section]
\newtheorem{lemma}[theorem]{Lemma}
\newtheorem{prop}[theorem]{Proposition}
\newtheorem{coro}[theorem]{Corollary}

\newcommand{\cP}{{\ensuremath{\mathcal P}} }

\newcommand{\cX}{{\ensuremath{\mathcal X}} }

\newcommand{\bbE}{{\ensuremath{\mathbb E}} }
\newcommand{\E}{{\ensuremath{\mathbb E}} }

\newcommand{\N}{{\ensuremath{\mathbb N}} }

\newcommand{\bbP}{{\ensuremath{\mathbb P}} }

\newcommand{\R}{{\ensuremath{\mathbb R}} }

\newcommand\I{\operatorname{I}}
\newcommand\J{\operatorname{J}}
\newcommand\U{\Lambda}
\newcommand\V{\operatorname{V}}
\renewcommand\H{\mathbf{H}}

\newcommand{\mc}[1]{{\mathcal #1}}

\newcommand{\mb}[1]{{\mathbf #1}}
\newcommand{\bb}[1]{{\mathbb #1}}
\newcommand{\eps}{\varepsilon}

\newenvironment{rostera}{\begin{enumerate}[\upshape(a)]}{\end{enumerate}}
\newfont{\indic}{bbmss12}
\def\un#1{\hbox{{\indic 1}$_{#1}$}}
\usepackage{color}
\definecolor{light}{gray}{.80}

\author{M. Mariani}
\author{L. Zambotti}
\begin{document}
\title[Large Deviations for Renewal Markov Processes]{Large deviations for the empirical measure of heavy tailed Markov renewal processes}

\date{}

\begin{abstract}
A large deviations principle is established for the joint law of the empirical measure and the flow measure of a renewal Markov process on a finite graph. We do not assume any bound on the arrival times, allowing heavy tailed distributions. In particular, the rate functional is in general degenerate (it has a nontrivial set of zeros) and not strictly convex. These features show a  behaviour highly different from what one may guess with a heuristic Donsker-Varadhan analysis of the problem.
\end{abstract}


\maketitle

\section{Introduction and main results}
\label{s:1}

\subsection{A motivating example} 
\label{ss:1.1}
Consider a finite collection $E$ of servers (or websites) in a given network, and a user who visits the servers and downloads data from them. It has been observed empirically that the size of downloaded files from the web has heavy tails (e.g.\ a polynomial decay at infinity), and thus a sharp statistical treatment of the matter can be difficult, see e.g.\ \cite{web}. 

We denote by $\psi_x$ the probability distribution on $]0,+\infty[$ of the size of files downloaded by a user visiting the server $x \in E$, and by $p_{x,y}$ the probability that a user visiting $x\in E$ will next visit $y\in E$. Let us suppose that the cost of downloading a file of size $\tau$ from the site $x$ is $f(x,\tau)$ and the cost of switching from server $x$ to server $y$ is $g(x,y)$; after downloading an amount $t$ of data, the servers $x_1,\ldots,\,x_{N_t} \in E$ will have been visited in sequence, with corresponding amounts of data $(\tau_1,\ldots,\,\tau_{N_t})$ downloaded respectively on each server. The cost of providing such a service (or the price asked for it) will depend on several factors, a reasonable form being
\begin{equation}
\label{e:measflow}
\sum_{i=1}^{N_t} \tau_{i} \, f(x_i,\tau_i) + \sum_{i=1}^{N_t} g(x_i,x_{i+1}).
\end{equation}
In the concrete example where the user is a www-crawler, the total time to gather data is indeed the "cost" of the service, so that $f(x,\tau)$ measures indeed the transmission performance of server $x$, while $g(x,y)$ may depend on servers lag times. The analysis of large deviations quantities \eqref{e:measflow} in the (relevant) limit where the total downloaded data are large $t\to +\infty$, is exactly equivalent to the large deviations of the empirical measures and flow considered below in this paper (however $t$ is hereafter interpreted as a time).


Markov processes with heavy tailed waiting times appear in several physical or biological models, see e.g.\ \cite{basbov, condpnas, hotscatterer}.
Large deviations principles for processes driven by renewal phenomena have been extensively analysed in the last decade, see e.g. \cite{duto, dumet, duro, ganesh, russ}. Typical techniques include sharp renewal estimates \cite[Chapter XIII]{asmussen} and so-called contraction principles via inversion maps \cite{russ, dumet}. In particular one can obtain large deviations principles for the renewal version of processes whose detailed asymptotic properties are known, using inversion maps \cite{russ, dumet}. In this paper, we consider a case in which these techniques do not work due to the presence of heavy tailed distributions of the arrival times of the renewal processes considered. To put it shortly, in our setting not only the usual inversion map is not continuous (and thus the contraction principle is not allowed), but a naive application of this strategy would suggest indeed a wrong result.

This paper has been inspired by \cite{bfg}, which considers the case of a countable state space $E$ and exponential waiting times, which make $(x_{N_t})_{t\geq 0}$ a Markov process.

\subsection{Main results}
\label{ss:1.2}
Let $E$ be a finite set, equipped with its discrete topology; elements of $E$ are denoted $x,y$. The set $[0,+\infty]$ will be equipped with any, compatible, complete separable metric (for instance, make it isometric to $[0,1]$); variables running on $[0,+\infty]$ will be denoted $s,\,t,\,\tau$. For a separable metric space $X$, hereafter $\mc P(X)$ denotes the set of Borel probability measures on $X$. For $\mu \in \mc P(X)$ and $f$ a measurable, $\mu$-integrable functions on $X$, $\mu(f)$ denotes the integral of $f$ with respect to $\mu$. $\mc P(X)$ is the equipped with the narrow topology, namely the weakest topology such that the maps $f \mapsto \mu(f)$ is continuous for any continuous bounded function $f$ on $X$. For $\mu,\,\nu \in \mc P(X)$, $\H(\nu\,|\,\mu)$ denotes the relative entropy of $\nu$ with respect to $\mu$:
\begin{equation}
\label{e:relent}
\H(\nu\,|\,\mu):=\sup_{\varphi \in C_{\mathrm{b}}(X)} \nu(\varphi) - \log \mu(e^{\varphi})
\end{equation}

The process we will next introduce is defined once the three following objects are given.
\begin{rostera}
\item A map $E\times E \ni (x,y) \mapsto p_{x,y} \in [0,1]$, which is an \emph{irreducible} Markov kernel on the \emph{finite} set $E$.
\item A map $E \ni x \mapsto \psi_x \in \mc P([0,+\infty])$, such that $\psi_x(\{0\})=\psi_x(\{+\infty\})=0$ for all $x\in E$. In particular no moment bound is assumed on $\psi_x$.
\item An arbitrary initial measure $\gamma \in \mc P(E)$.
\end{rostera}

We then consider the Markov renewal process $(X_k,\tau_{k+1})_{k\geq 0}$ defined on a probability space $(\Omega,\mathfrak F,\bbP)$ uniquely characterised by the two following properties.
\begin{itemize}
\item $(X_k)_{k\geq 0}$ is an irreducible Markov chain on 
$E$ with probability transition matrix $(p_{x,y})_{x,y\in E}$, and initial distribution $\gamma$ (for $X_0$).
\item $(\tau_i)_{i\geq 1}$ is a random sequence in $]0,+\infty[$, such that conditionally to $(X_k)$ the $\tau_i$ are independent and have distribution
\begin{equation*}
\bbP(\tau_i\in A \, | \, (X_k)_{k\geq 0}) = \psi_{X_{i-1}}(A), \qquad A\subset ]0,+\infty[
\end{equation*}
\end{itemize}
We refer to \cite[VII.4]{asmussen} for the presentation of the Markov renewal processes framework introduced above. Note that often one assumes the law of the $\tau_i$ conditioned to $(X_k)_{k \ge 0}$ to depend both on $X_{i-1}$ and $X_i$ (namely, one may replace $\psi_{X_{i-1}}$ with a $\psi_{X_{i-1},X_i}$ above). However, one can always reduce to the case here described by considering the Markov chain $(Y_k)_{k \ge 0}$, $Y_k:=(X_k,X_{k+1})$ instead. As we have no hypotheses on $\psi_x$, this doubling-variables procedure is fully compatible with our framework.

For $t \ge 0$ and $n \ge 0$ an integer, define the switching times $S_n$ and the number $N_t$ of switches up to time $t$ as
\begin{equation*}
S_n:=\sum_{i=1}^n \tau_i, \qquad
N_t:= \sum_{n=1}^\infty \un{(S_n\leq t)} = \inf\{n \geq 0: S_{n+1}>t \},
\end{equation*}
where we understand empty sums to vanish and $\inf \emptyset=+\infty$.

Fixed $t>0$, we then define the empirical measure $\mu_t \in \mc P(E \times [0,+\infty])$ of $(X_{N_s},\tau_{N_s})$ up to time $t$ by requiring for all $f\in C_{\mathrm{b}}(E \times [0,+\infty])$
\begin{equation}
\label{def:emp mes}
\begin{split}
   \mu_t(f) & :=\frac1t \int_0^t f(X_{N_s},\tau_{N_s+1}) \, ds \\ & =
    \frac{1}{t}\sum_{k=1}^{N_t}\tau_{k} \, f(X_{k-1},\tau_k)+\frac{t-S_{N_t}}{t}f(X_{N_t},\tau_{N_t+1}),
    \end{split}
\end{equation}
Similarly we define the empirical flow $Q_t \in C(E\times E; [0,+\infty[)$ as
\begin{equation}\label{def:empflow}
  Q_t(x,y) := \frac{1}{t} \sum_{k=1}^{N_t+1} \un{(X_{k-1}=x,X_{k}=y)}
      \qquad x,y\in E
    \end{equation}
Note that the process $(Z_t:=X_{N_t})_{t\ge 0}$ has a natural interpretation. Pick an initial datum $Z_0=X_0$ with law $\gamma$ on $E$, and next a time $\tau_1>0$ with law $\psi_{X_0}$. At time $\tau_1$, $Z_t$ jumps to $X_1$, (chosen  on $E$ with distribution $p_{X_0,\cdot}$). Now $Z$ spends a time $\tau_2$ (chosen with law $\psi_{X_1}$) at $X_1$, next jumping at $X_2$ at time $\tau_1+\tau_2$ and so on. Thus, $\mu_t$ is the joint empirical law of the process $Z_t$ and the associated inter-jumps times, while $Q_t(x,y)$ is the total number of jumps of $Z_t$ from $x$ to $y$ up to time $t$.

We want to investigate the large deviations of the joint law of $(\mu_t,Q_t)$ under the probability $\bb P$ of the renewal Markov chain $(X_k,\tau_{k+1})_{k \ge 0}$. Before stating our large deviations result, we need some further definition to introduce the spaces and the rate functional.

Let $\Lambda:= \mc P(E\times [0,+\infty]) \times C(E\times E;[0,+\infty[)$, which is a completely metrizable, separable topological space under the narrow topology of $\mc P(E\times [0,+\infty]) $ and the uniform topology of $C(E\times E;[0,+\infty[)$. Then define
$\xi \in C(E;[0,+\infty])$, and  $\Lambda_0 \subset \Lambda$ as
\begin{equation*}
\xi(x) := \sup \big\{ c\ge 0\,:\: \int \psi_x(d\tau) \,e^{c \tau} <+\infty \big\}
\end{equation*}
\begin{equation}
\label{def:U}
\begin{split}
 \Lambda_0:= \left\{  (\mu,Q)\in \Lambda\, :\:
\int_{[0,+\infty]}\!\!\!\!\!\!\! \mu(x,d\tau)\, \frac{1}{\tau} = \sum_{y\in E} Q(x,y)= \sum_{y\in E} Q(y,x), \, \forall  x\in E\right\}
\end{split}
\end{equation}
For $(\mu,Q) \in \Lambda_0$ define the Markov transition kernel $p^Q$ and the map $E \ni x \to \psi^\mu_x \in \mc P([0,+\infty])$ as
\begin{equation}
\label{overlineQ}
p^Q_{x,y} := \frac{Q(x,y)}{\sum_{z\in E} Q(x,z)}
\qquad x,y\in E
\end{equation}
\begin{equation*}
\psi^{\mu}_x(d\tau):=\frac{1}{\sum_{z\in E} Q(x,z)}\, \frac{1}{\tau}\, \mu(x,d\tau)
\qquad x\in E
\end{equation*}
where, if $Q(x,\cdot) \equiv 0$, the (irrelevant) choice $p^Q_{x,y} = p_{x,y}$, $\psi_x^\mu=\psi_x$ is understood.
Then define $\I \colon \Lambda \to [0,+\infty]$ as
\begin{equation}\label{I}
\I(\mu,Q) = 
\begin{cases}
&\displaystyle{ \!\!\!\!\!\!\!\!\!\! \!\!\!\!\! \sum_{x\in E} \int_{[0,+\infty]} 
\mu(x,d\tau)\,\left[ 
\frac{\H\big( p^Q_{x,\cdot} \, |\,  p_{x,\cdot}  \big) 
+ \H\big( \psi^\mu_x \, | \, \psi_x  \big)}{\tau} +  \xi(x) \un {\{\infty\}}(\tau) \right]
} 
\\  &  \textrm{if $(\mu,Q)\in \Lambda_0$} \\
\\  +\infty & \textrm{otherwise}.
\end{cases}
\end{equation}
where hereafter we adopt the convention $0\cdot\infty=0$. A variational carachterization of $\I$ is given in Proposition~\ref{l:3.4}.

\begin{prop}\label{prop:bon fonction de I}
The functional $\mathrm{I}$ is \emph{good}, namely for each $M>0$ the set $\{(\mu,Q):\mathrm{I}(\mu,Q)\leq M\}$ is compact.
\end{prop}

The following theorem is the main result of this paper.
\begin{theorem}
\label{thm:main result}
The law $\mathbb{P} \circ(\mu_t,\,Q_t)^{-1}$ of $(\mu_t,\,Q_t)$ under $\bb P$ satisfies a large deviations principle as $t\to +\infty$, with speed $t$ and good rate $\I$. Namely, for each closed set $\mc C \subset \Lambda$ and each open set $\mc O \subset \Lambda$
\begin{equation}
\label{eq:upper bound}
\limsup_{t\to +\infty}\frac{1}{t}\log\mathbb{P}((\mu_t,Q_t)\in \mathcal{C})\leq -\inf_{(\mu,Q)\in \mathcal{C}}\mathrm{I}(\mu,Q)
\end{equation}
\begin{equation}
\label{eq:lower bound}
\liminf_{t\to +\infty}\frac{1}{t} \log\mathbb{P} ((\mu_t,Q_t)\in \mathcal{O})\geq -\inf_{(\mu,Q)\in \mathcal{O}}\mathrm{I}(\mu,Q)
\end{equation}
\end{theorem}

\subsection{Large Deviations for the empirical measure of $X_{N_\cdot}$} 
\label{ss:1.3}
Define $\I_1 \colon \mc P(E) \to [0,+\infty]$ as
\begin{equation}\label{def:I chapeau mu}
\mathrm{I}_1(\nu)=\inf\{\mathrm{I}(\mu,Q), \, (\mu,Q)\in \Lambda \,:\: \mu(dx,[0,+\infty])=\nu(dx)   \}.
\end{equation}
Theorem~\ref{thm:main result} yields, via a standard application of the so-called contraction principle \cite[Chapter~4.2.1]{Zeitouni}, the following corollary. It is however remarkable that $\mathrm{I}_1$ admits an explicit expression, see Proposition~\ref{exprI_1} below, generalizing \cite{duto} to our framework.
\begin{coro}\label{coro:result}
The law of $\frac 1t \int_{[0,t[} \delta_{X_{N_s}}\,ds$, the empirical measure of the process $X_{N_\cdot}$ satisfies a large deviations principle as $t\to +\infty$, with speed $t$ and good rate $\mathrm{I}_1$.
\end{coro}

The rest of the paper is organized as follows. In section~\ref{s:2} the properties of the rate functional are discussed, and in sections \ref{sec:upbound}-\ref{sec:lowbound} the proofs of the large deviations principle are provided together with further remarks. In section~\ref{sec:contr} results concerning the contraction principle in section~\ref{ss:1.3} are proved.

\section{The functional $\I$}
\label{s:2}
In this section some \emph{deterministic} results concerning the functional $\I$ are established, and
Proposition~\ref{prop:bon fonction de I} is proved. We first remark an immediate identity.

\begin{lemma}\label{lem:algfact}
For all $a>0$ and $\pi,\psi\in \mathcal{P}([0,+\infty[)$ such that $\pi(1/\tau)<+\infty$
\begin{equation*}
\begin{split}
\pi(1/\tau) \,\H(\tilde{\pi}\,|\,\psi) & = \sup_{h}(\pi(h/\tau)-\pi(1/\tau)\log\psi(e^h))
\\ &
= \sup_{h:\psi(e^h)=a}(\pi(h/\tau)-\pi(1/\tau)\log\psi(e^h))=\sup_{h:\psi(e^h)=1}\pi(h/\tau)
\end{split}\end{equation*}
where suprema are taken over $h \in C_b(\mathbb{R}_+^\ast)$ and
\begin{equation*}
\tilde \pi(d\tau):= \frac{\tfrac{1}{\tau}\pi(d\tau)}{\pi(1/\tau)}
\end{equation*}
\end{lemma}
\begin{proof}
Recall \eqref{e:relent}
Suppose now that $\psi(e^h)=a>0$ and set $h^a:=h-\log a$. Then
the quantity
\begin{equation*}
\sup_{h:\psi(e^h)=a}(\pi(h/\tau)-\pi(1/\tau)\log\psi(e^h))
\end{equation*}
does not depend on $a>0$ and thus
\begin{equation*}
\sup_{h:\psi(e^h)=a}(\pi(h/\tau)-\pi(1/\tau)\log\psi(e^h))=\sup_{h:\psi(e^h)=1}\pi(h/\tau)=\pi(1/\tau)
\, \H(\tilde{\pi} \, | \, \psi)
\end{equation*}
where all suprema are taken over $h\in C_b(\mathbb{R}_+^\ast)$.
\end{proof}

Given $c,\,M \in C(E)$ such that $ 0 \le c(x) \le \xi(x)$, with $c(x)<\xi(x)$ for all $x$ such that $\xi(x)>0$, and given  $\varphi \in C_{\mathrm{b}}(E \times [0,+\infty])$,  $(x,\tau) \mapsto \varphi_x(\tau)$, define  $h^{\varphi,c,M} \colon E\times [0,+\infty]\mapsto [0,+\infty[$ as
\begin{equation}\label{h}
h_x^{\varphi,c,M}(\tau) = \frac{\varphi_x(\tau)}\tau + c(x) \, \un{]M(x),+\infty]}(\tau), \quad x\in E, \ \tau\in \, ]0,+\infty],
\end{equation}
Notice that $h_x^{\varphi,c,M}$ is l.s.c. on $]0,+\infty]$.
Then define
\begin{equation}\label{Gamma}
\begin{split}
\Gamma :=\Big\{ (h,H) \, : \: &  h=h^{\varphi,c,M}\,\text{for $c,\,M,\,\varphi$ as above and s.t.\
$\psi_{x} \left(e^{\tau h_x} \right)< $1 }
\\ & H\colon E\times E\mapsto\R \,\,\text{such that} \sum_{z\in E} p_{x,z} \, e^{H(x,z)} < 1, \ \forall \, x\in E\Big\}.
\end{split}
\end{equation}

For all $(h,H)\in \Gamma$ we denote by
$I_{h,H}: \Lambda\mapsto \mathbb{R}$ the functional
\begin{equation}\label{def:I_F varphi}
\begin{split}
& \,I_{h,H}(\mu,Q):= \sum_{x,y\in E}  Q(x,y) \left( H(x,y) - \log\sum_z q_{x,z} \, e^{H(x,z)}\right)\\
&\,+\sum_{x\in E}\left(\int_{]0,+\infty]} \mu(x,d\tau) \, h_x(\tau)-\sum_{y}Q(x,y)\log \int_{]0,+\infty[} \psi_{x}(ds) \, e^{s\, h_{x}(s)} \right).
\end{split}
\end{equation}

\begin{prop}
\label{l:3.4}
For $(\mu,Q)\in \Lambda$
\begin{equation} \label{e:muf}
  \I(\mu,Q) = \sup_{(h,H)\in\Gamma} \I_{h,H}(\mu,Q).
\end{equation}
In particular, $\I$ is convex.
\end{prop}

\begin{proof}
Let first $(\mu,Q)\in \Lambda_0$ and let us consider $(h,H)\in\Gamma$. Then, by the well known properties of the relative entropy and since $\xi_x\geq c_x$, we find easily that $\I(\mu,Q)\geq \I_{h,H}(\mu,Q)$, so that by the arbitrariness of $(h,H)\in\Gamma$
\begin{equation*}
 \I(\mu,Q) \ge \sup_{(h,H)\in\Gamma} \I_{h,H}(\mu,Q).
\end{equation*}
Now, let us prove the converse inequality. We have for all $(h,H)\in\Gamma$ with $h=h^{\varphi,c,M}$
\begin{equation*}
\begin{split}
& \sup_{\varphi} \sup_{c} \limsup_{M\to+\infty} \, \langle \mu, h^{\varphi,c,M} \rangle 
\\ &= \sup_{\varphi} \sum_x \left(\mu(x,1/\tau) \, \tilde\mu(x,\varphi_x-\log \psi_x(e^{\varphi_x})) + \mu(x,+\infty)\, \xi_x\right) \\
&= \sup_{\varphi: \psi_x(e^{\varphi_x})=1} \sum_x \left(\mu(x,1/\tau) \, \tilde\mu(x,\varphi_x) + \mu(x,+\infty)\, \xi_x\right)\\
& = \sum_x \left(\mu(x,1/\tau) \, \H\left(\tilde\mu(x,\cdot)\, | \, \psi_x\right) +  \mu(x,+\infty)\, \xi_x\right).
\end{split}
\end{equation*}
On the other hand
\begin{equation*}
\sup_{H: \langle q_{x,\cdot},e^H\rangle<1} \, \langle Q,H\rangle  = \sum_x \Big( \sum_y Q(x,y)\Big)
\H\left( \overline Q(x,\cdot) \, | \, q_{x,\cdot} \right).
\end{equation*}
Therefore
\begin{equation*}
\sup_{(h,H)\in\Gamma} \I_{h,H}(\mu,Q)
\geq \I(\mu,Q).
\end{equation*}

Finally, let $(\mu,Q) \in \Lambda \setminus \Lambda_0$; we want then to show that $\sup_{(h,H)\in\Gamma} \I_{h,H}(\mu,Q)=+\infty$.
Since $(\mu,Q) \in \U \setminus \U_0$, then $\mu(x,1/\tau)<\sum_y Q(x,y)$ for some $x\in E$.
Then for $H(x,\cdot)\equiv 0$ and $h_x(\tau)=-M/\tau$ we have
\begin{equation*}
I_{h,H}(\mu,Q) \geq M\Big(\sum_{y}Q(x,y)-\mu(x,1/\tau) \Big) \to +\infty.
\end{equation*}
\end{proof}

We turn now to show the proof of Proposition \ref{prop:bon fonction de I}. 
\begin{proof}[Proof of Proposition \ref{prop:bon fonction de I}]
Let $(\mu_n,Q_n)\subset \mathcal{Y}$ be a sequence of measures such that
\begin{equation}\label{eq:I fini}
    \varlimsup_{n\rightarrow+\infty}\mathrm{I}(\mu_n,Q_n)<+\infty
\end{equation}
We need to prove that $(\mu_n,Q_n)_n$ is precompact (coercivity of $\mathrm{I}$) and that for any limit point $(\mu,Q)$ of $(\mu_n,Q_n)$, $\liminf_{n\rightarrow+\infty}\mathrm{I}(\mu_n,Q_n)\geq\mathrm{I}(\mu,Q)$ (lower semi-continuity of $\mathrm{I}$). Notice that (\ref{eq:I fini}) implies $(\mu_n,Q_n)\subset \U_0$ for $n$ large enough, i.e.
\begin{equation}\label{eq:relation mu Q}
\mu_n(x,1/\tau) = \sum_{y\in E}Q_n(x,y), \quad \sum_{y\in E}(Q_n(x,y)-Q_n(y,x))=0, \qquad \forall \, x\in E.
\end{equation}

\noindent\textit{Coercivity of $\mathrm{I}$}. By the bound (\ref{eq:I fini})
\begin{equation}\label{eq:a,b}
\varlimsup_{n\rightarrow +\infty}
\mu_n(x,1/\tau) \, \H(\tilde\mu_n(x,\cdot) \, | \, \psi_x) <+\infty.
\end{equation}
From any subsequence, we can extract a sub-subsequence along which
\begin{equation*}
\H(\tilde\mu_n(x,\cdot) \, | \, \psi_x)\to b, \qquad \mu_n(x,1/\tau)\to a,
\end{equation*}
with $a,b\in[0,+\infty]$. Let us set for simplicity $\phi:=\psi_x$ and $\pi_n:=\mu_n(x,\cdot)$. If $b=+\infty$, then $a=0$, thus let us suppose that $b<+\infty$; then by the inequality $t\log t\geq -e^{-1}$ and by
Jensen's inequality we get
\begin{equation}\label{entrine}
\frac{1}{e}\phi(\R_+^*\setminus E)+\H\big(\tilde\pi_n\, \big|\, \phi\big)\geq
\int_E f_n\ln f_n\,d\phi\geq\tilde\pi_n(E)\log\frac{\tilde\pi_n(E)}{\phi(E)},
\end{equation}
for any Borel $E\subset \R$, where $\tilde\pi_n=f_n\phi$.
Choosing $E=]0,\eps[$ with $0<\eps <+\infty$, we find that
$\sup_n\tilde\pi_n(E)\to 0$ as $\eps\to 0$. Therefore
$(\tilde\pi_n)_n$ is tight in $\cP(]0,+\infty])$. Therefore, up to passing to a further subsequence,
$\tilde\pi_n\rightharpoonup \tilde\pi\in \cP(]0,+\infty])$. By lower semicontinuity of the relative entropy $\H\big(\cdot\, \big|\, \phi\big)$, we obtain that $\H\big(\tilde\pi\, \big|\, \phi\big)<+\infty$
and therefore $\tilde\pi(\{+\infty\})=\phi(\{+\infty\})=0$.
We claim also that $\pi_n(1/\tau)\to\pi(1/\tau)<+\infty$. Notice that
\begin{equation*}
\int_{]0,\eps[}\frac1\tau\, \pi_n(d\tau) = \pi_n(1/\tau)\, \tilde\pi_n(]0,\eps[)
\end{equation*}
and, by boundedness of $(\pi_n(1/\tau))_n$ and tightness of $(\tilde\pi_n)_n$
we obtain that
\begin{equation*}
\lim_{\eps\to 0} \sup_n \int_{]0,\eps[}\frac1\tau\, \pi_n(d\tau) = 0.
\end{equation*}
Therefore, by uniform integrability, we obtain that
\begin{equation*}
\pi(1/\tau)=\lim_n \pi_n(1/\tau)<+\infty.
\end{equation*}

It follows that necessarily $a<+\infty$. Let us denote
\begin{equation*}
K_{M}:=\left \{(\mu,Q)\in \Lambda \, :\,
 \mu(1/\tau)\leq \sum_{x,y}Q(x,y)\leq M \right\}
\end{equation*}
Since $(\mu_n,Q_n)\in\U_0$, then $\sup_n Q_n(x,E) = \sup_n \mu(x,1/\tau)<+\infty$, then $(\mu_n,Q_n)\in K_M$ for $M$ large enough. It is not difficult to verify the compactness of $K_M$ for all $M<+\infty$, therefore we can conclude compactness of $(\mu_n,Q_n)_n$ in $\mathcal{Y}$.

\medskip
\noindent\textit{Semi-continuity of $\mathrm{I}$}. Let $(\mu_n,Q_n)_n\subset \Lambda$ be such that $(\mu_n,Q_n)\to(\mu,Q)$ in $\Lambda$. We want to show that
\begin{equation*}
\I(\mu,Q) \leq \varliminf_n \I(\mu_n,Q_n).
\end{equation*}
Since $\U$ is closed in $\Lambda$ and $\I\equiv+\infty$ on $\Lambda\setminus\U$, we can suppose that $(\mu,Q)\in\U$.


Let us consider $(h,H)\in\Gamma$, see \eqref{Gamma}. Since $h$ is s.l.c. and $\U$ is closed in $\Lambda$, then $\I_{h,H}$ is also l.s.c. on $\Lambda$. Then by \eqref{e:muf} we obtain
\begin{equation*}
\I_{h,H}(\mu,Q)\leq \varliminf_n \I_{h,H}(\mu_n,Q_n) \leq \varliminf_n \I(\mu_n,Q_n)
\end{equation*}
and by the arbitrariness of $(h,H)$
\begin{equation*}
\I(\mu,Q) = \sup_{(h,H)\in\Gamma} \I_{h,H}(\mu,Q)\leq \varliminf_n \I(\mu_n,Q_n).
\end{equation*}
\end{proof}

\begin{lemma}\label{lem:omega dense}
Let $\U$ be defined by \eqref{def:U}, and let
\begin{equation}\label{U00}
\begin{split}
\U_{00}:= \{  & \ (\mu,Q)\in\U :  \I(\mu,Q)<+\infty, \\ & \ \mu(x,+\infty) = 0 \ {\rm and} \ \mu(x,]0,+\infty[)>0, \ \forall \, x\in E,
\\ & \ (\overline Q(x,y))_{x,y\in E} \ {\rm defines \ an \ irreducible \ transition \ matrix \ on  \ } E \},
\end{split}
\end{equation}
where $\overline Q$ is defined as in \eqref{overlineQ}.

Then $\U_{00}$ is $\I$-dense in $\U$, namely for all $(\mu,Q)\in \U$ with $\I(\mu,Q)<+\infty$, there exists a sequence $(\mu_n,Q_n)_n$ in $\U_{00}$ such that 
\begin{equation*}
(\mu_n,Q_n)\to (\mu,Q), \qquad \lim_n\mathrm{I}(\mu_n,Q_n)= \I(\mu,Q).
\end{equation*}
\end{lemma}
\begin{proof}
Let us start by proving that the following set 
\begin{equation*}
\begin{split}
\U_1:= \{  & \ (\mu,Q)\in\U :  \I(\mu,Q)<+\infty, \ \mu(x,]0,+\infty[)>0, \ \forall \, x\in E,
\\ & \ (\overline Q(x,y))_{x,y\in E} \ {\rm defines \ an \ irreducible \ transition \ matrix \ on  \ } E \},
\end{split} 
\end{equation*}
is $\I$-dense in $\U$.
For any $x\in E$, let $A_x$ be a bounded Borel subset of $]0,+\infty[$ such that $\psi_x(A_x)>0$. We can set
\begin{equation*}
\mu^0(x,d\tau) := \frac{\nu_x \, \tau \, \psi_x(d\tau \,|\, A_x)}{Z}, \quad Z:= \sum_{y\in E} \nu_y \int  \tau \, \psi_y(d\tau \,|\, A_y), \quad Q^0(x,y):=\frac{\nu_x\, q_{x,y}}{Z},
\end{equation*}
where $Z<+\infty$ since every $A_y$ is assumed to be bounded. Since $(q_{xy})_{x,y\in E}$ is assumed to be irreducible, $Q^0$ is so as well; at the same time, it is easy to see that $\I(\mu^0,Q^0)<+\infty$ and therefore
$(\mu^0,Q^0)\in\U_1$. We note in particular that $\mu^0(x,\cdot)\ll \psi_x(\cdot)$ and $Q^0(x,\cdot)\ll q_{x\cdot}$ for all $x\in E$. Now for
any $(\mu,Q)\in\U$ with $\I(\mu,Q)<+\infty$, we set
\begin{equation*}
\mu^\eps := \eps\, \mu \, + (1-\eps)\, \mu^0, \qquad
Q^\eps := \eps\, Q+ (1-\eps)\, Q^0, \qquad \eps\in[0,1].
\end{equation*}
Then $(\mu^\eps,Q^\eps)\in\U_0$ and by convexity
\begin{equation*}
\I(\mu^\eps,Q^\eps) \leq \eps\I(\mu,Q)+(1-\eps)\I(\mu^0,Q^0)
\end{equation*}
and we obtain that $(\mu^\eps,Q^\eps)\to(\mu,Q)$ and $\I(\mu^\eps,Q^\eps)\to\I(\mu,Q)$ as $\eps\to 1$.

Let us now show that $\U_{00}$ is $\I$-dense in $\U_1$.
We note first that
\begin{equation}\label{liminf}
\varlimsup_M \frac1M \log \psi_x([M,+\infty[) = -\xi_x, \qquad \forall \, x\in E.
\end{equation}
Indeed, by the exponential Markov inequality, for all $0\leq c<\xi$
(we drop the subscript $x$ for simplicity of notation)
\begin{equation*}
\psi([M,+\infty[) \leq \psi(e^{c(\tau-M)}), \qquad \forall \, M>0,
\end{equation*}
so that, by taking the supremum over $c<\xi$.
\begin{equation*}
\varlimsup_M \frac1M \log \psi([M,+\infty[) \leq -\xi.
\end{equation*}
Now, if $\xi<+\infty$, suppose that the inequality is strict, i.e. there exists $c>\xi$ and $M_c$ large enough such that 
\begin{equation*}
\psi([M,+\infty[) \leq e^{-cM}, \qquad \forall \, M\geq M_c.
\end{equation*}
Then we have for all $c'<c$
\begin{equation*}
\begin{split}
\psi(e^{c'\tau}) & = \int_0^{+\infty} \psi(e^{c'\tau}>t) \, dt   \leq 1+ c'\int_0^\infty
\psi(\tau>s) \, e^{c's} \, ds \\ &
\leq 1+e^{c'M_c}+\int_{M_c}^{+\infty} c'e^{-(c-c')s} \, ds <+\infty.
\end{split}
\end{equation*}
In this way, the definition of $\xi$ is contradicted and we have proved \eqref{liminf}. Therefore there exists a sequence $(M_n(x))_{n\in\N}$ such that $M_n(x)\to+\infty$ and 
\begin{equation*}
\lim_n \frac1{M_n(x)} \log \psi_{x}([M_n(x),+\infty[) = -\xi_{x},
\end{equation*}

Let us now fix $(\mu,Q)\in \U_1$ and $x\in E$. If $\mu(x,+\infty)=0$, then $\mu_n(x,\cdot):=\mu(x,\cdot)$,
$Q_n(x,\cdot):=Q(x,\cdot)$. If $\mu(x,+\infty)>0$, then, since $\I(\mu,Q)<+\infty$, 
necessarily $\xi_{x}<+\infty$ and therefore $\psi_{x}([M,+\infty[)>0$ for all $M>0$.
Moreover, since $(\mu,Q)\in \U_1$ we have $\mu(x,1\tau)>0$ and finite; finally, $\mu(x,\cdot)$ must be of the form
\begin{equation*}
\mu(x,d\tau) = \rho(\tau) \, \psi_x(d\tau) + a\, \delta_{+\infty}(d\tau).
\end{equation*}
Now, denoting $I_n:=[M_n(x),+\infty[$,
we set for $\alpha_n,\beta_n\geq 0$
\begin{equation*}
\mu_n(x,\cdot) := \left(\alpha_n \,  \rho(\tau) \, + \beta_n \, a \, \frac{\un{I_n}}{\psi_x(I_n)}
\right)\psi_x(d\tau).
\end{equation*}
Now we want to fix $\alpha_n,\beta_n$ such that
\begin{equation*}
\mu_n(x,]0,+\infty[) = \mu(x,]0,+\infty]), \qquad \mu_n(x,1/\tau) = \mu(x,1/\tau),
\end{equation*}
i.e. 
$\alpha_n \,  \mu(x,]0,+\infty[) + \beta_n \, a = \mu(x,]0,+\infty[) + a$ and
\begin{equation*}
\alpha_n \,  \mu(x,1/\tau) + \beta_n \, a \, \psi_x(1/\tau \, | \, I_n)
= \mu(x,1/\tau).
\end{equation*}
This linear system has a unique solution $(\alpha_n,\beta_n)$ if 
\begin{equation*}
\mu(x,]0,+\infty[) \, \psi_x(1/\tau \, | \, I_n) -  \mu(x,1/\tau) \ne 0,
\end{equation*}
which is true for $n$ large enough, since $\psi_x(1/\tau \, | \, [M,+\infty[)\to 0$ as $M\to+\infty$.
For the same reason, $\alpha_n\to 1$ and
$\beta_n\to 1$. In particular, this shows that $\mu_n\rightharpoonup\mu$ on $E\times ]0,+\infty]$ as $n\to+\infty$. Setting $Q_n:=Q$,
since $\mu_n(x,1/\tau) = \mu(x,1/\tau)$ for all $x\in E$ by construction, we have that $(\mu_n,Q_n)\in\U_{00}$.

Notice now that, since $\mu_n(x,]0,+\infty[) = \mu(x,]0,+\infty])$,
\begin{equation*}
\tilde\mu_n(x,\cdot) = \alpha_n \, \tilde\mu(x,\cdot) + (1-\alpha_n) \, \nu_n(\cdot),
\end{equation*}
where 
\begin{equation*}
\nu_n(d\tau) = \frac1\tau \, \frac{\un{I_n}}{\psi_x(1/\tau; I_n)} \, \psi_x(d\tau).
\end{equation*}
Now, by the convexity of $\H(\cdot \, | \, \psi_x)$, 
\begin{equation*}
\H(\tilde\mu_n(x,\cdot) \, | \, \psi_{x}) \leq \alpha_n \, \H(\tilde\mu(x,\cdot)\, | \, \psi_{x}) +(1-\alpha_n) \, \H(\tilde\nu_n \, | \, \psi_{x}).
\end{equation*}
Now, recalling that $I_n=[M_n(x),+\infty[$,
\begin{equation*}
\begin{split}
(1-\alpha_n) \, \H(\tilde\nu_n \, | \, \psi_{x}) & =  \beta_n a \, \frac{\psi_{x}( 1/\tau \,| \, I_n)}{\mu(x,1/\tau)}
\, \log\left(\frac1{\psi_{x}( I_n)}\right)
\\ & \leq \frac{\beta_n a}{\mu(x,1/\tau)} \, \frac1{M_n(x)} \, \log\left(\frac1{\psi_{x}( I_n)}\right)
\to \frac{a}{\mu(x,1/\tau)} \, \xi_x, 
\end{split}
\end{equation*}
as $n\to+\infty$. Then, it follows easily that
\begin{equation*}
\begin{split}
\limsup_n \I(\mu_n,Q_n) & \leq  \I(\mu,Q).
\end{split}
\end{equation*}
Since $\I$ is lower semi-continuous, the proof is complete.
\end{proof}

\section{Upper bound}\label{sec:upbound}

We define
\[
\begin{split}
& \U:= \left\{  (\mu,Q)\in \Lambda  :  
\mu(x,1/\tau) \leq \sum_y Q(x,y), \, \sum_{z}(Q(x,z)-Q(z,x))=0, \, \forall \, x\in E\right\}
\\ & \U_0:= \left\{(\mu,Q)\in \U  \ : \  \mu(x,1/\tau) = \sum_{y\in E}Q(x,y), \, \forall x\in E \right\}.
\end{split}
\]
\begin{lemma}
\label{closed}
Let $\mc O \subset \Lambda$ be open and $\mc O \supset \Lambda_0$. Then
\begin{equation}
\label{eq:expon}
 \limsup_{t\to +\infty}\tfrac{1}{t}\log\mathbb{P} \big( (\mu_t,Q_t) \notin \mc O \, \big)=-\infty.
\end{equation}
\end{lemma}
\begin{proof}
$\U$ is closed by Fatou's lemma. By the definitions \eqref{def:emp mes} and
\eqref{def:empflow}, we have
\begin{equation*}
\mu_t(x,1/\tau)= \frac{1}{t} \sum_{k=1}^{N_t} \un{(X_{k-1}=x)}+\frac{t-S_{N_t}}{\tau_{N_t+1}}
\end{equation*}
\begin{equation*}
\sum_{y\in E}Q_t(x,y)=\frac{1}{t} \sum_{k=1}^{N_t+1} \un{(X_{k-1}=x)}
\end{equation*}
Namely
\begin{equation*}
0 \le \sum_{y\in E}Q_t(x,y)- \mu_t(x,1/\tau) = \frac{1}{t} \,\frac{S_{N_t+1}-t}{\tau_{N_t+1}} \le \frac 1t
\end{equation*}
Moreover
\begin{equation*}
\Big| \sum_{y\in E}(Q_t(x,y)-Q_t(y,x))\Big| = \Big| \frac{\un{(X_0=x)}-\un{
(X_{N_t+1}=x)}}t\Big| \le \frac{1}{t}.
\end{equation*}
Then \eqref{eq:expon} follows, since $(\mu_t,Q_t) \notin \V$ for $t$ large enough.
\end{proof}

\subsection{Exponential tightness}
In order to prove the upper bound we have to show the exponential tightness of the sequence $\{\mathbb{P}\circ (\mu_T,Q_T)^{-1}\}_{T>0}$. To do so, we need to show that
\begin{equation*}
\inf_{\mc K } \varlimsup_{t\to+\infty}
\frac{1}{t}  \log {\bf P}_t (\mc K^c)=-\infty.
\end{equation*}
where ${\mc K }$ varies among all compact subsets of $\Lambda$.
\begin{lemma}
\label{l:3.2}
Setting $K_{M}:=\left \{(\mu,Q)\in \Lambda \, :\,
 \mu(1/\tau)\leq \sum_{x,y}Q(x,y)\leq M \right\}$ then $K_{M}$ is compact in $\Lambda$ and
 \begin{equation*}
\lim_{M\to+\infty} \varlimsup_{t\rightarrow +\infty}\frac{1}{t}\log \mathbb{P}((\mu_t,Q_t)\in K_{M}^c) = -\infty.
\end{equation*}
\end{lemma}
\begin{proof}
We recall that $\mu_t(1/\tau)\leq \sum_{x,y}Q_t(x,y)$, see the proof of Lemma \ref{closed}. Moreover
\begin{equation*}
\mu_t(1/\tau) = \frac{N_t}t + \frac1t \, \frac{t-S_{N_t}}{\tau_{N_t+1}} \leq \frac{N_t+1}t,
\end{equation*}
so that $\{\mu_t(1/\tau)>M\}\subset \{N_t+1>n\}= \{S_n\leq t\}$ for $n=\lfloor Mt\rfloor$. Then 
\begin{equation*}
\begin{split}
\{(\mu_t,Q_t)\in K^c_M\}=&\left\{\mu_t(1/\tau)>M\right\} = \{N_t+1> Mt+1\}=\{S_{\lfloor Mt+1\rfloor}\leq t\}.
\end{split}
\end{equation*}
By the Markov inequality,
\begin{equation*}
\mathbb{P}((\mu_t,Q_t)\in K_{M}^c)\leq\mathbb{P}(S_{\lfloor Mt+1\rfloor}\leq t)\leq \mathbb{E}(e^{t-S_{\lfloor Mt+1\rfloor}})
\leq e^{t+\lfloor Mt+1\rfloor\log c}
\end{equation*}
where $c:=\sup_{y\in E}\psi_y(e^{-\tau})<1$. Then we have
\begin{equation*}
\limsup_{t\rightarrow +\infty}\frac{1}{t}\log\mathbb{P}((\mu_t,Q_t)\in K_{M}^c)\leq 1+(M+1)\log c
\end{equation*}
and this tends to $-\infty$ as $M\to+\infty$. Compactness of $K_{M}$ is standard.
\end{proof}

\subsection{Change of probability}
To prove the upper bound, fix $(h,H)\in \Gamma$. We define
\begin{equation}\label{hH}
q^H(x,y) := \frac{q(x,y) \, e^{H(x,y)}}{\sum_z q(x,z) \, e^{H(x,z)}}, \qquad
\psi^h_x(d\tau) := \frac{e^{\tau \, h_{x}(\tau)} \, \psi_{x}(d\tau)}{\int e^{s \, h_{x}(s)} \, \psi_{x}(ds)}
\end{equation}
and we call $\bbP^{(h,H)}$ the law of the renewal Markov process $(X_k,\tau_{k+1})_{k\geq0}$
with transition probability $(q^H,\psi^h)$.
Then,
\begin{equation*}
\begin{split}
&\left. \frac{1}{t}\log\frac{d\mathbb{P}^{(h,H)}}{d\mathbb{P}} \right|_{\sigma( (X_k,\tau_{k+1})_{k\leq N_t+1})} \\
=&\frac{1}{t}\sum_{i=1}^{N_t+1}\log\frac{e^{H(X_{i-1},X_{i})}}{\sum_z q_{X_{i-1},z}\, e^{H(X_{i-1},z)}}\,
+\frac{1}{t}\sum_{i=1}^{N_t+1}\log\frac{e^{\tau_i \, h_{X_{i-1}
}}(\tau_i)}{\int e^{\tau_i \, h_{X_{i-1}
}}(s) \, \psi_{X_{i-1}
}(ds)}\\
=&\sum_{x,y} Q_t(x,y) \left( H(x,y) - \log\sum_z q_{x,z} e^{H(x,z)}\right)\\
&+\sum_{x}\left(\int_{]0,+\infty]} \mu_t(x,d\tau) \, h_x(\tau)-\sum_y Q_t(x,y)\log \int _{]0,+\infty[}\psi_{x}(ds) \, e^{s h_{x}(s) } \right) \\ & +\frac{\tau_{N_t+1}-t+S_{N_t}}{t} \, h_{X_{N_t}}(\tau_{N_t+1})
\\ = & \I_{h,H}(\mu_t,Q_t)+\frac{\tau_{N_t+1}-t+S_{N_t}}{t} \, h_{X_{N_t}}(\tau_{N_t+1}),
\end{split}
\end{equation*}
where $I_{h,H}$ is defined in \eqref{def:I_F varphi} above. Now, recall that $h$ takes the form
\eqref{h}. Then on the event $\{n=N_t+1,X_{n-1}=x\}$
\begin{equation*}
\frac{\tau_{N_t+1}-t+S_{N_t}}{t} h_{X_{N_t}}(\tau_{N_t+1})= \frac{\tau_n-t+S_{n-1}}t\left(\frac{\varphi_{x}(\tau_{n})}{\tau_{n}}+ c_x \un{]M_x,+\infty]}(\tau_n) \right) \geq -\frac{\|\varphi\|_\infty}t
\end{equation*}
since $0\leq\tau_n-t+S_{n-1}\leq\tau_{n}$ and $c_x\geq 0$. Therefore, for $\mc A$ measurable subset of $\Lambda$ and for $(h,H) \in \Gamma$
\begin{equation*}
\begin{split}
    \frac{1}{t} \log {\bf P}_t(\mc A) &
\leq      \frac{1}{t} \log
        \E\left(  \un {\mc A}(\mu_t,Q_t) \, e^{-t \I_{h,H}(\mu_t,Q_t)+\|\varphi\|_\infty} \, \left.\frac{d\mathbb{P}^{(h,H)}}{d\mathbb{P}} \right|_{\sigma( (X_k,\tau_{k+1})_{k\leq N_t+1})}\right)
      \\ &
\le \frac{1}{t} \log\left[   e^{ -t\inf_{(\mu,Q) \in \mc A}\I_{h,H}(\mu,Q)+\|\varphi\|_\infty} \,
\E\left( \left.\frac{d\mathbb{P}^{(h,H)}}{d\mathbb{P}} \right|_{\sigma( (X_k,\tau_{k+1})_{k\leq N_t+1})} \right) \right]\\
& = -\inf_{(\mu,Q) \in \mc A}
\I_{h,H}(\mu,Q)+\frac{\|\varphi\|_\infty}{t},
    \end{split}
\end{equation*}
and therefore
\begin{equation}\label{e:ldmeas}
      \varlimsup_{t \to + \infty} \frac{1}{t} \log \mb P_t(\mc A)
\le  -\inf_{(\mu,Q) \in \mc A} \I_{h,H}(\mu,Q).
\end{equation}
For $M>0$, $g\in C_c(]0,+\infty])$ for all $x\in E$, $G:E^2\mapsto \R$ and $\delta>0$, let
\begin{equation*}
K_{M,g,G,\delta}:= \{ (\mu,Q)\in K_{M} \, : \, \exists \, (\mu',Q')\in \U, |\mu(g)-\mu'(g)|+
|Q(G)-Q'(G)|\leq \delta\},
\end{equation*}
where $K_{M}$ is the compact set defined in Lemma \ref{l:3.2}, and
\begin{equation*}
  R_{M,g,G,\delta}:= -\varlimsup_{t \to +\infty} \frac{1}{t} \log {\bf P}_t(K_{M,g,G,\delta}^c).
\end{equation*}
Let now $\mc O$ be
an open subset of $\Lambda$. Then applying
\eqref{e:ldmeas} for $\mc A=\mc O \cap K_{M,g,G,\delta}$
  \begin{equation*}
    \begin{split}
  \varlimsup_{t\to+\infty} \frac{1}{t} \log {\bf P}_t(\mc O) &  \le
      \varlimsup_{t\to+\infty} \frac{1}{t} \log\big[2 \max({\bf P}_t
      (\mc O \cap K_{M,g,G,\delta}),{\bf P}_t(K_{M,g,G,\delta}^c) )\big]
\\ &
\le \max\left(-\inf_{(\mu,Q) \in \mc O \cap K_{M,g,G,\delta}} \I_{h,H}(\mu,Q), -R_{M,g,G,\delta}\right)
\\ & =
-\inf_{(\mu,Q) \in \mc O \cap K_{M,g,G,\delta}} \I_{h,H}(\mu,Q) \wedge R_{M,g,G,\delta}
    \end{split}
\end{equation*}
which can be restated as
\begin{equation}\label{e:ldopen}
  \varlimsup_{t\to +\infty}  \frac{1}{t} \log {\bf P}_t(\mc O) \le
- \inf_{(\mu,Q) \in \mc O} \I_{h,H,M,g,G,\delta}(\mu,Q)
\end{equation}
for any open set $\mc O$,
where the functional $\I_{h,H,M,g,G,\delta}$ is defined as
\begin{equation*}
  \I_{h,H,M,g,G,\delta}(\mu,Q):=
  \begin{cases}
    \I_{h,H}(\mu,Q) \wedge R_{M,g,G,\delta} & \text{if $(\mu,Q) \in K_{M,g,G,\delta}$}
\\ +\infty & \text{otherwise}.
  \end{cases}
\end{equation*}
Since $h$ is lsc and $K_{M,g,G,\delta}$ is closed,
then $\I_{h,H,M,g,G,\delta}$ is lsc.
By minimizing \eqref{e:ldopen} over $\{h,H,M,g,G,\delta\}$ we obtain
\begin{equation*}
\varlimsup_{t\to +\infty}  \frac{1}{t} \log {\bf P}_t(\mc O) \le
- \sup_{h,H,M,g,G,\delta}\inf_{(\mu,Q) \in \mc O} \I_{h,H,M,g,G,\delta}(\mu).
\end{equation*}
Since $\mc O$ is arbitrary, by applying the minimax lemma
\cite[Appendix 2.3, Lemma 3.3]{KL}, we get that for any compact set $\mc K$
\begin{equation*}
\varlimsup_{t\to +\infty}  \frac{1}{t} \log {\bf P}_t(\mc K) \le
- \inf_{(\mu,Q) \in \mc K} \sup_{h,H,M,g,G,\delta} \I_{h,H,M,g,G,\delta}(\mu,Q)
\end{equation*}
i.e.\ $({\bf P}_t)_{t\geq 0}$ satisfies
a large deviations upper bound on compact sets with speed $t$ and rate functional
$\tilde{\I}:\Lambda\mapsto[0,+\infty]$ given by
\begin{equation*}
\tilde{\I}(\mu,Q):=  \sup_{h,H,M,g,G,\delta} \I_{h,H,M,g,G,\delta}(\mu,Q).
\end{equation*}
By Lemma~\ref{closed} we have $\cap_{g,G,\delta}K_{M,g,G,\delta}\subset
\U$, so that $\tilde{\I}(\mu,Q)=+\infty$
if $(\mu,Q)\notin\U$. We claim now that
\begin{equation*}
\lim_{M\to +\infty} R_{M,g,G,\delta}=+\infty, \qquad \forall \, g,G,\delta.
\end{equation*}
Indeed, by the definition we see that $K_M\cap \U\subset K_{M,g,G,\delta}$. Therefore ${\bf P}_t(K_{M,g,G,\delta}^c)\leq 2\max\{{\bf P}_t(K_{M}^c), {\bf P}_t(\U^c)\}$, and we conclude using \eqref{eq:expon} first and then Lemma~\ref{l:3.2}.
Therefore for all $(\mu,Q)\in\Lambda$
\begin{equation*}
\tilde{\I}(\mu,Q)\ge  \sup_{(h,H)\in\Gamma} \, \tilde\I_{h,H}(\mu,Q),
\end{equation*}
where
\begin{equation*}
  \tilde\I_{h,H}(\mu,Q):=
  \begin{cases}
    \I_{h,H}(\mu,Q) & \text{if $(\mu,Q) \in \U$},
\\ \\
+\infty & \text{otherwise}.
  \end{cases}
\end{equation*}
Thus $\tilde{\I}(\mu,Q) \ge \I(\mu,Q)$ by Proposition~\ref{l:3.4}. Therefore
$({\bf P}_t)_{t\geq 0}$ satisfies a large deviations upper bound with rate $\I$ on
compact sets. By Lemma~\ref{l:3.2} and \cite[Lemma 1.2.18]{Zeitouni},
$({\bf P}_t)_{t\geq 0}$ satisfies the full large deviations upper bound on closed sets.

\section{Laws of large numbers}\label{lln}
We prove now an auxiliary law of large numbers to be used in the proof of the lower bound. Hereafter we make the dependence on the distribution of $X_0$ explicit, writing $\E_\nu$ if $\nu$ is the law of $X_0$, whenever $\nu\neq\gamma$.
\begin{prop}\label{prop:lln2}
Suppose that
\begin{equation*}
\E_\nu(\tau_1) = \sum_y \nu_y\, \psi_y(\tau) < +\infty.
\end{equation*}
Then, for all $x,y,z\in E$, under $\bbP_x$-a.s. 
\begin{equation*}
\lim_{n\to+\infty} \frac{S_n}n \to \E_\nu(\tau_1), \quad
\lim_{t\to+\infty}  \mu_t(y,d\tau) = \frac{\nu_y}{ \E_\nu(\tau_1)}  \, \tau \,  \psi_y(d\tau),
\quad
\lim_{t\to+\infty}  Q_t(y,z) = \frac{ \nu_y \,  q_{yz}}{ \E_\nu(\tau_1)}.
\end{equation*}
\end{prop}
For any $y\in E$ we denote 
\begin{equation*}
\phi^y_{1}:=\mathrm{inf}\{\ell>0:X_{\ell-1}=y \}, \qquad \phi^y_{k+1}:=\mathrm{inf}\{\ell> \phi^y_{k}:X_{\ell-1}=y\}, 
\quad k\geq 1.
\end{equation*} 
Moreover we define
\begin{equation*}
   N_t^y=\sum_{k=1}^{+\infty}\un{(S_k\leq t, X_{k-1}=y)}=\sum_{k=1}^{+\infty}\un{(S_{\phi^y_{k}}\leq t)}, \qquad \forall \, y\in E,
\end{equation*}
i.e. the number of times the process $(X_{k})_{k=0,\ldots,N_t-1}$ visits the site $y$,
and 
\begin{equation*}
M_n^y:=\sum_{k=0}^{n-1}\un{(X_{k}=y)}=\sum_{k=1}^{+\infty}\un{({\phi^y_{i}}\leq n)}, \qquad \forall \, y\in E.
\end{equation*}
i.e. the number of times the process $(X_{k})_{k=0,\ldots,n-1}$ visits the site $y$.
\begin{lemma}\label{prop:1}
For any $x,y\in E$, under $\bbP_x$ the sequence $(\tau_{\phi^y_{k}})_{k\geq 1}$ is a i.i.d. sequence
with common distribution $\psi_y(\cdot)$.
\end{lemma}
\begin{proof}
Setting $\mathscr{F}^X:=\sigma(X_k, k\geq 0)$, we have that $({\phi^y_{k}})_{k\geq 1}$ is $\mathscr{F}^X$-measurable. Now, conditionally on
$\mathscr{F}^X$, the sequence $(\tau_j)_{j\geq 1}$ is independent. But conditionally on 
$\mathscr{F}^X$, $\tau_{\phi^y_{k}}$ has law $\psi_y(\cdot)$. Therefore, the conditional law of
$(\tau_{\phi^y_{k}})_{k\geq 1}$ given $\mathscr{F}^X$ is the law of a i.i.d. sequence with common
distribution $\psi_y(\cdot)$. Since this conditional law does not depend on $(X_k, k\geq 0)$, the result is proved.
\end{proof}
\begin{lemma}\label{prop:1.5}
For all $x\in E$, $\mathbb{P}_x$-a.s.
\begin{equation*}
\lim_{n\rightarrow+\infty} \frac{S_n}{n} = \E_\nu(\tau_1). 
\end{equation*}
\end{lemma}
\begin{proof}
We can see that
\begin{equation*}
\frac{S_n}{n}=\frac{1}{n}\sum_{y\in E} \sum_{i=1}^n \un{(X_{i-1}=y)} \, \tau_i=\frac{1}{n}\sum_{y\in E}\sum_{i=1}^{M_n^y}\tau_{\phi^y_{i}}=\sum_{y\in E}\frac{M_n^y}{n}\frac{1}{M_n^y}\sum_{i=1}^{M_n^y}\tau_{\phi^y_{i}}.
\end{equation*}
By the ergodic theorem, for any $f:E\mapsto\R$, $\bbP_x$-a.s.
\begin{equation*}
\frac{1}{n}\sum_{k=0}^{n-1}f(X_k)\rightarrow\nu(f)=\sum_{y\in E}f(y)\, \nu_y.
\end{equation*}
Then we have $\bbP_x$-a.s.
\begin{equation*}
    \frac{M_n^y}{n}=\frac{1}{n}\sum_{k=0}^{n-1}\un{(X_k=y)}\rightarrow \nu_y.
\end{equation*}
Thus, by law of the large numbers, we have the conclusion.
\end{proof}

\begin{lemma}\label{prop:2}
For any $x\in E$, $\mathbb{P}_{x}$-a.s.
\begin{equation}\label{eq:3}
\lim_{t\rightarrow+\infty}\frac{N^{y}_{t}}{t}=\frac{\nu_y}{\,\E_\nu(\tau_1)}, \quad \forall \,y\in E,
\qquad \lim_{t\rightarrow+\infty}\frac{N_{t}}{t}=\frac1{\,\E_\nu(\tau_1)}.
\end{equation}
\end{lemma}
\begin{proof}
Under $\bbP_x$, the sequence
$(S_{\phi^y_{k+1}}-S_{\phi^y_{k}})_{k\geq 1}$ is i.i.d. and
by the renewal theorem, $\mathbb{P}_{x}$-a.s.
\begin{equation*}
\lim_{t\rightarrow\infty}\frac{N^{x}_{t}}{t}=\frac{1}{\mathbb{E}_{x}(S_{\phi^y_{2}}-S_{\phi^y_{1}})}.
\end{equation*}
Now, by the strong Markov property of $(X_k)_{k\geq 0}$
\begin{equation*}
\mathbb{E}_x(S_{\phi^y_{2}}-S_{\phi^y_{1}})=\mathbb{E}_x\left(\sum^{\phi^y_{2}}_{i=\phi^y_{1}+1}\tau_{i}\right)
= \mathbb{E}_y\left(\sum^{\phi^y_{1}}_{i=1}\tau_{i}\right) = \sum_z\mathbb{E}_y\left(\sum^{\phi_{1}^y}_{i=1}\un{(X_{i-1}=z)}\,\tau_i\right).
\end{equation*}
Now, since $\mathbb{E}_y(\tau_i \, | \, (X_k)_{k\geq 0} ) = \psi_{X_{i-1}}(\tau)$, we have
\begin{equation*}
\begin{split}
\mathbb{E}_y\left(\sum^{\phi_{1}^y}_{i=1}\un{(X_{i-1}=z)}\,\tau_i\right) & = \mathbb{E}_y\left(\sum^{\phi_{1}^y}_{i=1}\un{(X_{i-1}=z)}\,\mathbb{E}_y(\tau_i \, | \, (X_k)_{k\geq 0} )\right) 
\\ & = \mathbb{E}_y\left(\sum^{\phi_{1}^y}_{i=1}\un{(X_{i-1}=z)}\right) \psi_z(\tau) =   \frac{\nu_z}{\nu_y} 
\, \psi_z(\tau)
\end{split}
\end{equation*}
by \cite[ Corollary I 3.6]{asmussen}. Therefore
\begin{equation*}
\mathbb{E}_x(S_{\phi^y_{2}}-S_{\phi^y_{1}}) =\frac1{\nu_y} \sum_z \, {\nu_z}\,\psi_z(\tau) 
\end{equation*}
and the proof of the first assertion is complete. Now, it is enough to note that 
\begin{equation*}
N_t  = \sum_{y\in E} N^y_t, \qquad t\geq 0,
\end{equation*}
and this concludes the proof.
\end{proof}

\begin{proof}[Proof of Proposition \ref{prop:lln2}]
Recalling the definition \eqref{def:emp mes} of the empirical measure $\mu_t$, we have for $y\in E$
\begin{equation}\label{eq:mu t law}
\begin{split}
\mu_t(y,\cdot)&=\frac{1}{t} \sum_{i=1}^{N_t^{y}}\tau_{\phi^y_{i}}\, \delta_{\tau_{\phi^y_{i}}}+\frac{t-S_{N_t}}{t} \, \un{(X_{N_t}=y)} \, \delta_{\tau_{N_t+1}}.
\end{split}
\end{equation}
By Lemmas \ref{prop:1.5} and \ref{prop:2}, $\bbP_x$-a.s.
\begin{equation*}
\lim_{t\to+\infty} \frac{S_{N_t}}{t} = \lim_{t\to+\infty} \frac{S_{N_t}}{N_t}\, \frac {N_t}{t} 
= 1 \, \Longrightarrow \, \lim_{t\to+\infty} \frac{t-S_{N_t}}{t} = 0.
\end{equation*}
On the other hand, by Lemma \ref{prop:1} and the law of large numbers, for all bounded measurable $f:\R_+\mapsto\R$ we have $\bbP_x$-a.s.
\begin{equation*}
\lim_n \frac{1}{n} \sum_{i=1}^{n}\tau_{\phi^y_{i}}\, f(\tau_{\phi^y_{i}}) = \int_0^\infty \tau\, f(\tau)\, \psi_y(d\tau)
\end{equation*}
and therefore by Lemma \ref{prop:2}
\begin{equation*}
\lim_{t\rightarrow+\infty}\frac{1}{t}\sum_{i=1}^{N^y_t}\tau_{\phi^y_{i}}\, f(\tau_{\phi^y_{i}})= 
\lim_{t\rightarrow+\infty}\frac{N^y_t}{t}\, \frac1{N^y_t}\sum_{i=1}^{N^y_t}\tau_{\phi^y_{i}}\, f(\tau_{\phi^y_{i}})= \frac{\nu_y}{\,\E_\nu(\tau_1)} \, \int_0^\infty \tau\, f(\tau)\, \psi_y(d\tau).
\end{equation*}
Therefore for all $g:E\times\,]0,+\infty[$ bounded and measurable we have
\begin{equation*}
\lim_{t\rightarrow+\infty} \mu_t(g) = \sum_{y\in E} \frac{\nu_y}{\,\E_\nu(\tau_1)} \int_0^\infty \tau\, g(y,\tau)\, \psi_y(d\tau).
\end{equation*}
We prove now the almost sure convergence of the empirical flow $Q_t(y,z)$. We have
\begin{equation*}
Q_t(y,z)=\frac1{t} \sum_{i=1}^{N_t^{y}}\un{(X_{\phi^y_{i}+1}=z)} , \qquad y,z\in E.
\end{equation*}
Setting $Y_i:=(X_{\phi^y_{i}+j})_{j=0,\ldots,\phi^y_{i+1}-\phi^y_{i}}$, then by the strong Markov property
under $\bbP_x$ the sequence $(Y_i)_{i\geq 1}$ is i.i.d. and its law is equal to the law of $(X_{j})_{j=0,\ldots,\phi^y_{1}}$ under $\bbP_y$. Then by the law of large numbers, $\bbP_x$-a.s.
\begin{equation*}
\lim_n \frac1{n} \sum_{i=1}^{n}\un{(X_{\phi^y_{i}+1}=z)} = \bbP_y(X_1=z) = q_{yz}.
\end{equation*}
Therefore, by Lemma \ref{prop:2}, $\bbP_x$-a.s.
\begin{equation*}
\lim_{t\rightarrow+\infty} Q_t(y,z) = \lim_{t\rightarrow+\infty} \frac{N^y_t}t \, \frac1{N^y_t}
\sum_{i=1}^{N_t^{y}}\un{(X_{\phi^y_{i}+1}=z)} = \frac{\nu_y \, q_{yz}}{\,\E_\nu(\tau_1)}.
\end{equation*}
The proof is complete.
\end{proof}

\section{Lower bound}\label{sec:lowbound}
For the proof of the lower bound, let us denote by $\mathbf{P}_t$ the law of $(\mu_t,Q_t)$. Then it is well known that it is enough to show the following
\begin{prop}\label{prop:lower bound}
For every $(\mu,Q)\in \U_0$ and $t>0$, there exists a family of probability measures $\mathbf{Q}_t$ such that $\mathbf{Q}_t\rightharpoonup\delta_{(\mu,Q)}$ as $t\uparrow +\infty$ and
\begin{equation*}
\limsup_{t\rightarrow+\infty}\frac{1}{t}\H(\mathbf{Q}_t\, | \,\mathbf{P}_t)\leq \mathrm{I}(\mu,Q).
\end{equation*}
\end{prop}
Indeed, if Proposition \ref{prop:lower bound} is proved, then we reason as follows. Let $(\mu,Q)\in\U_0$ and let
$\mathcal{V}$ be an open neighborhood of $(\mu,Q)$ in the weak topology. Then
\begin{equation*}
\begin{split}
\log\mathbb{P}((\mu_t,Q_t)\in\mathcal{V})&=\log\int_{\mathcal{V}}\frac{d\mathbf{P}_t}{d\mathbf{Q}_t}d\mathbf{Q}_t=\log\left(\frac{1}{\mathbf{Q}_t(\mathcal{V})}\int_{\mathcal{V}}\frac{d\mathbf{P}_t}{d\mathbf{Q}_t}d\mathbf{Q}_t\right)+\log\mathbf{Q}_t(\mathcal{V})\\
&\geq\frac{1}{\mathbf{Q}_t(\mathcal{V})}\int_{\mathcal{V}}\log\frac{d\mathbf{P}_t}{d\mathbf{Q}_t}d\mathbf{Q}_t+\log\mathbf{Q}_t(\mathcal{V}).
\end{split}
\end{equation*}
by using Jensen's inequality. Now, since $x \log x \geq -e^{-1}$ for all $x\geq 0$, we obtain
\begin{equation*}
\begin{split}
\log \mathbf{P}_t(\mathcal{V})&\geq \frac{1}{\mathbf{Q}_t(\mathcal{V})}\left(-\H(\mathbf{Q}_t\, | \,\mathbf{P}_t)+\int_{\mathcal{V}^{c}}\log\frac{d\mathbf{P}_t}{d\mathbf{Q}_t}\frac{d\mathbf{Q}_t}{d\mathbf{P}_t}d\mathbf{P}_t\right)+\log\mathbf{Q}_t(\mathcal{V})\\
&\geq \frac{1}{\mathbf{Q}_t(\mathcal{V})}\left(-\H(\mathbf{Q}_t\, | \,\mathbf{P}_t)-e^{-1}\right)+\log\mathbf{Q}_t(\mathcal{V}).
\end{split}
\end{equation*}
Since $(\mu,Q)\in \U_0$, $\mathbf{Q}_t\rightharpoonup\delta_{(\mu,Q)}$ and $\mathcal{V}$ is open, then $\mathbf{Q}_t(\mathcal{V})\rightarrow 1$ as $t\rightarrow+\infty$. We obtain
\begin{equation*}
\liminf_{t\rightarrow +\infty}\frac{1}{t}\log \mathbf{P}_t(\mathcal{V})\geq -\limsup_{t\rightarrow +\infty}\frac{1}{t}\H(\mathbf{Q}_t\, | \,\mathbf{P}_t)\geq -\mathrm{I}(\mu,Q).
\end{equation*}
Therefore, for any open set $\mathcal{O}$ and for any $(\mu,Q)\in \U_0$,
\begin{equation*}
\liminf_{t\rightarrow +\infty}\frac{1}{t}\log \mathbf{P}_t(\mathcal{O})\geq  -\mathrm{I}(\mu,Q).
\end{equation*}
and by optimizing over $(\mu,Q)\in \mathcal{O}$ we have the lower bound.

\begin{proof}[Proof of Proposition \ref{prop:lower bound}]
Let us first suppose that $(\mu,Q)\in \U_{00}$ as defined in \eqref{U00}, i.e. $(\mu,Q)\in \U_0$, $\mu(x,+\infty)=0$,
\begin{equation*}
Z_x := \mu(x,1/\tau) = \sum_{z} Q(x,z)>0, \qquad \forall \, x\in E,
\end{equation*}
$\overline Q(x,y):= Q(x,y)/Z_x$ defines an irreducible probability transition matrix on $E$ and
$\mu(x,]0,+\infty[)>0$ for all $x\in E$, so that $\mu(x,1/\tau)>0$ for all $x\in E$. Moreover, since $(\mu,Q)\in \U_0$, then
$\mu(x,1/\tau)<+\infty$ for all $x\in E$ and in particular $\mu(x,\cdot)\ll \psi_x(\cdot)$ and $Q(x,\cdot)\ll q_{x\cdot}$.
Then it makes sense to define for $x,y\in E$ 
\begin{equation*}
H(x,y) := \log\left(\frac{Q(x,y)}{Z_x \, q_{xy} }\right) \un{(q_{xy}>0)},
\qquad h_x(\tau) := \frac1\tau
\log \left(
\frac{\mu(x,d\tau)}{Z_x \, \tau \, \psi_x(d\tau)} \right).
\end{equation*}
In this way, the probability kernels $q^H$ and $\psi^h$ defined in \eqref{hH} become
\begin{equation*}
q^H_{xy} = \frac1{Z_x} \, {Q(x,y)} =\overline Q(x,y), \qquad \psi_x^h(d\tau) =  \frac{1}{Z_x} \, \frac 1\tau \, \mu(x,d\tau).
\end{equation*}

We denote by $\bbP^{(h,H)}_x$ the law of a Markov renewal process $(X_k,\tau_{k+1})_{k\geq 0}$
with transition kernel $(q^H,\psi^h)$ and initial state $X_0=x$. By the irreducibility of $q^H=\overline Q$, the unique invariant measure $\nu^H$ of $q^H$ is given by 
\begin{equation*}
\nu^H_x:=\frac{Z_x}Z, \qquad Z := \sum_{y\in E} Z_y,
\end{equation*}
since we have
\begin{equation*}
\sum_{x\in E} \nu^H_x \, q^H_{xy} = \frac 1Z \sum_x Q(x,y) = \frac 1Z \sum_x Q(y,x)
= \frac{Z_y}Z = \nu^H_x.
\end{equation*}
Moreover, since $\mu(y,+\infty)=0$ for all $y\in E$ by the assumption $(\mu,Q)\in U_{00}$,
\begin{equation*}
\E^{(h,H)}_{\nu^H}(\tau_1) = \sum_{y\in E} \nu^H_y \, \psi^h_y(\tau) = \frac1Z \sum_{y\in E}
Z_y \, \frac1{Z_y} \int_0^\infty \tau \, \frac1\tau \, \mu(y,d\tau) = \frac 1Z < +\infty.
\end{equation*}
Then, we have by Proposition \ref{prop:lln2} below that $\bbP^{(h,H)}_x$-a.s.
\begin{equation*}
\mu_t(y,d\tau)ÔøΩ\rightharpoonup \, {Z} \,  \nu_y^H \, \tau \, \psi_y^h(d\tau) = \mu(y,d\tau), \qquad
Q_t(y,z) \to Z \, {\nu_y^H \, q_{y,z}^H} = Q(y,z)
\end{equation*}
as $t\to+\infty$. We set
$T_t:=\lfloor (1+\delta)Zt\rfloor$ and we denote by
$\bbP_{t,\delta}$ the law of $(X_k,\tau_{k+1})_{k\geq 0}$ under which
\begin{enumerate}
\item $(X_k,\tau_{k+1})_{k=0,\ldots,T_t}$ is a Renewal Markov process with transition rates $(\hat q, \hat \psi)$ and $X_0=x$ a.s.
\item conditionally on $(X_k,\tau_{k+1})_{k=0,\ldots,T_t}$,  $(X_k,\tau_{k+1})_{k\geq T_t}$  is a Renewal Markov process with transition rates $(q,\psi)$.
\end{enumerate}
Then we denote by $\mathbf{Q}_{t,\delta}$ the law of $(\mu_t,Q_t)$ under $\bbP_{t,\delta}$. Let us prove first that
\begin{equation}\label{sieg}
\lim_{\delta\downarrow 0}\lim_{t\uparrow +\infty} {\bf Q}_{t,\delta} = \delta_{(\mu,Q)}.
\end{equation}
By Lemma \ref{prop:1.5}, under $\bbP^{(h,H)}_x$ we have a.s.
\begin{equation*}
\lim_{t\to+\infty} \frac{S_{T_t}}t =
\lim_{t\to+\infty} \frac{S_{T_t}}{T_t} \,\frac{T_t}t=
\E^{(h,H)}_{\nu^H}(\tau_1)\, Z\, (1+\delta)=1+\delta.
\end{equation*}
However $S_{T_t}$
has the same law under $\bbP^{(h,H)}_x$ and under $\bbP_{t,\delta}$, so
for any $\delta>0$
\begin{equation}\label{mire}
\lim_{t\to+\infty} \bbP_{t,\delta}\left(S_{T_t}\leq t\right)=
\lim_{t\to+\infty} \bbP^{(h,H)}_x\left(\frac{S_{T_t}}t\leq 1\right)=0.
\end{equation}
Therefore, if we set
\begin{equation*}
D_{t,\delta}:=\left\{S_{T_t}>t\right\}
\end{equation*}
then, by \eqref{mire} we obtain that
for all $\delta>0$
\begin{equation}
\label{e:aa}
\lim_{t\to+\infty} \bbP_{t,\delta}\left(D_{t,\delta} \right)=1.
\end{equation}
We recall that $\{S_n>t\}=\{N_t+1\leq n\}$.
Therefore on $D_{t,\delta}$ we have $N_t+1\leq T_t$ and setting for any $f\in C_b(\cX)$ and $\eps>0$
\begin{equation*}
A^f_{t,\eps}:=\left\{ |\mu_t(f)-\mu(f)|>\eps, \ \sup_{x,y\in E}|Q_t(x,y)-Q(x,y)|>\eps \right\},
\end{equation*}
we have
\begin{equation*}
\bbP_{t,\delta}(A^f_{t,\eps})\leq\bbP^{(h,H)}_x(A^f_{t,\eps}\cap D_{t,\delta})+
\bbP_{t,\delta}(D_{t,\delta}^c).
\end{equation*}
By Proposition \ref{prop:lln2}
\begin{equation*}
\lim_{t\uparrow+\infty} \bbP^{(h,H)}_x\left(|\mu_t(f)-\mu(f)|>\eps, \sup_{x,y\in E}|Q_t(x,y)-Q(x,y)|>\eps \right)=0,
\end{equation*}
which, in view of \eqref{e:aa}, implies \eqref{sieg}. Now we estimate the relative entropy
\begin{equation*}
\H(\mathbf{Q}_{t,\delta}\,|\,\mathbf{P}_t)  \leq \H(\bbP_{t,\delta}\, | \, \mathbb{P}_x)
= t\, \bbE_{t,\delta}(J_{t,\delta}) = t\, \bbE^{(h,H)}_x(J_{t,\delta})
\end{equation*}
where
\begin{equation*}
J_{t,\delta} := \frac{1}{t}\sum_{i=1}^{T_t-1} \left(H(X_{i-1},X_{i})
+\tau_{i} \, h_{X_{i-1}
}(\tau_{i}) \right).
\end{equation*}
By the ergodic theorem, we have
\begin{equation*}
\begin{split}
& \lim_{t\to+\infty} \bbE^{(h,H)}_x\left(\frac{1}{t}\sum_{i=1}^{T_t-1} \left(H(X_{i-1},X_{i})
+\tau_{i} \, h_{X_{i-1}
}(\tau_{i}) \right)\right)
\\ & = \frac{1+\delta}{Z} \sum_x \nu^H_x \left( \sum_y q^H_{x,y} \, H(x,y) + \psi^h_x(\tau\, h_x)\right)
\\ & = (1+\delta)\left(\sum_{x,y} Q(x,y) \, H(x,y) + \sum_x \mu(x,h_x)\right) \leq (1+\delta)\I(\mu,Q)
\end{split}
\end{equation*}
since
\begin{equation*}
\sum_y q_{x,y} \, e^{H(x,y)}= \psi_x(e^{\tau \, h_x}) = 1, \qquad \forall \, x\in E.
\end{equation*}
Therefore
\begin{equation*}
\lim_{\delta \downarrow 0} \varlimsup_{t\uparrow +\infty}
\frac 1t \H({\bf Q}_{t,\delta} \,| \, {\bf P}_t)  \le  \I(\mu,Q).
\end{equation*}
Then there exists a map $t\mapsto \delta(t)>0$ vanishing as $t \uparrow +\infty$ such that
${\bf Q}_t:= {\bf Q}_{t,\delta(t)} \to \delta_{(\mu,Q)}$ and $\varlimsup_t t^{-1}
\H({\bf Q}_t \,| \, {\bf P}_t) \le \I(\mu,Q)$.
\end{proof}

\section{Contraction principles}\label{sec:contr}

Let us consider now the empirical measure of the process 
$(X_{N_t})_{t\geq 0}$ alone, namely
\begin{equation*}
  \pi_t(x) :=\frac1t \int_0^t \un{(X_{N_s}=x)} \, ds  =
    \frac{1}{t}\sum_{k=1}^{N_t}\tau_{k} \, \un{(X_{k-1}=x)}+\frac{t-S_{N_t}}{t} \, \un{(X_{N_t}=x)}, \qquad t>0. 
\end{equation*}
We want to obtain a LDP for $(\pi_t)_t$ as $t\to+\infty$. To this aim, we need the large-deviation functional of the empirical measure of the Markov chain $(X_k)_{k\geq 0}$: if
\[
\zeta_n(x) = \frac1n \sum_{k=0}^{n-1} \un{(X_k=x)}, \qquad x\in E,
\]
then the law of $(\zeta_n)_n$ satisfies a LDP in the probability measures on $E$ with good rate function
\[
I_{DV} (\zeta) = \sup_{u\in\,]0,+\infty[^E} \sum_{x\in E}
\zeta_x \, \log\left( \frac{u_x}{\sum_y p_{x,y}u_y} \right) =
\sum_{x\in E}
\zeta_x \, \log\left( \frac{u_x^*}{\sum_y p_{x,y}u_y^*} \right) 
\]
where $u^*\in \,]0,+\infty[^E$ is the only vector such that
\[
\zeta_y = \sum_{x\in E} \zeta_x \, p^*(x,y), \qquad {\rm where}
\qquad p^*(x,y) := \frac{p(x,y)\, u_y^*}{\sum_z p_{x,z}u_z^*}
\]
i.e. $u^*$ makes $\zeta$ an invariant measure for $p^*$: see \cite[Theorem IV.6, IV.7]{denhollander}.
\begin{prop}\label{exprI_1}
The law of $\pi_t$ satisfies a LDP in as $t\to+\infty$ with good rate functional
\[
\I_1(\pi) = \inf_{\zeta} \left( I_{DV}(\zeta) + \sum_{x\in E} \zeta_x \,\Lambda_x^*(\pi_x/\zeta_x)
\right).
\]
\end{prop}
\begin{proof}
For all $(\mu,Q)\in\Lambda)$, let us denote $\pi(x):=\mu(x,]0,+\infty])$, $x\in E$. Then
$(\pi(x))_{x\in E}$ defines a probability measure on $E$. The map $(\mu,Q)\mapsto\pi$
is continuous and we obtain therefore by the contraction principle a LDP for the law of $(\pi_t(x))_{t>0}$
as $t\to+\infty$, where
The LDP rate function is given by
\begin{equation*}
\J(\pi) = \inf\{ \I(\mu,Q) : \, (\mu,Q)\in\U_0, \, \mu(x,]0,+\infty]) = \pi(x), \ x\in E\}.
\end{equation*}
In some particular case, a formula for $\J$ has already been computed, see for instance
Corollary 5.2 of Duffy and Torrisi \cite{duto}. 

\subsection{Computations}
We want to compute the infimum of
\[
\I(\mu,Q) =  \sum_{x\in E} \int_{[0,+\infty]} 
\mu(x,d\tau)\,\left[ 
\frac{\H\big( p^Q_{x,\cdot} \, |\,  p_{x,\cdot}  \big) 
+ \H\big( \psi^\mu_x \, | \, \psi_x  \big)}{\tau} +  \xi(x) \un {\{\infty\}}(\tau) \right]
\]
over all $(\mu,Q)\in\U_0$ such that $\mu(x,]0,+\infty]) = \pi(x)$, for all $x\in E$. We can write
\[
\begin{split}
\I(\mu,Q) = & \sum_{x,y\in E} Q(x,y) \, \log\left(\frac{p^Q(x,y)}{p(x,y)}\right) +
\sum_{x\in E} \int_{]0,+\infty[} \frac1\tau \, \mu(x,d\tau) \log\left(\frac{d\psi^\mu_x}{d\psi_x}(\tau)\right)
\\ & + \sum_{x\in E} \xi(x)\, \mu(x,+\infty).
\end{split} 
\]
We can suppose that $\mu(x,+\infty)=0$ for all $x\in E$. Now we can compute the infimum over $(\zeta(x), x\in E)$ of the infimum of $\I(\mu,Q)$ over $(\mu,Q)\in\U_0$ such that
$\mu(x,]0,+\infty]) = \pi(x)$ and $\mu(x,1/\tau)=\zeta(x)$, for all $x\in E$. Note that $\pi$ is a probability measure while $\zeta$ is positive but not necessarily normalized.
Using an analog of \cite[Theorems IV.6 and IV.7]{denhollander}, we obtain that the infimum of
\[
I^2_p(Q):=\sum_{x,y\in E} Q(x,y) \, \log\left(\frac{p^Q(x,y)}{p(x,y)}\right)
\]
over $Q$ such that $\sum_y Q(x,y)=\sum_yQ(y,x)=\zeta(x)$, for all $x\in E$, is equal to 
\[
I_{DV}(\zeta) := \sup_{u>0} \sum_x \zeta_x \, \log\left(\frac{u_x}{\sum_y p(x,y) \, u_y}\right),
\]
the unique optimal $u_y^*$ is such that 
\[
\zeta_x = \sum_y \zeta_y \, \frac{p(y,x) \, u_x^*}{\sum_z p(y,z) \, u_z^*}, \qquad
\forall \, x\in E,
\]
and the optimal $Q^*$ is
\[
Q^*(x,y) = \zeta_x \, \frac{p(x,y) \, u_y^*}{\sum_z p(x,z) \, u_z^*}, \qquad x,y\in E.
\]
Indeed, a simple computation shows that
\[
I^2_p(Q) = I^2_p(Q^*) + \H\left( p^Q \, | \, p^{Q^*} \right).
\]
Notice that in the classical Donsker-Varadhan case, $\zeta$ should be a probability measure, which is not necessarily the case here.

Therefore we are reduced to compute the infimum of
\[
\begin{split}
 & I_{DV}(\zeta) + 
\sum_{x\in E} \int_{]0,+\infty[} \frac1\tau \, \mu(x,d\tau) \log\left(\frac{d\psi^\mu_x}{d\psi_x}(\tau)\right)
\end{split} 
\]
over all $\zeta$ and $\mu$ such that $\mu(x,1/\tau)=\zeta_x$, $\mu(x,\R_+)=\pi_x$. This is equivalent to compute the infimum of
\[
I_{DV}(\zeta) + \sum_{x\in E} \zeta_x \, \H\left(\phi_x\,|\,\psi_x\right)
\]
over all $\zeta>0$ and $\phi_x\in\cP(]0,+\infty[)$ such that $\phi_x(\tau)=\pi_x/\zeta_x$. Now
\[
\inf\{\H\left(\phi_x\,|\,\psi_x\right) : \, \phi_x(\tau)=\pi_x/\zeta_x\}
= \Lambda_x^*(\pi_x/\zeta_x),
\]
where $\Lambda_x^*$ is the Legendre transform of $\Lambda_x(\theta):=\log(\psi_x(e^{\theta\tau}))$, $\theta\in\R$. Then we are reduced to
\[
\inf_{\zeta} \left( I_{DV}(\zeta) + \sum_{x\in E} \zeta_x \,\Lambda_x^*(\pi_x/\zeta_x)
\right).
\]
\end{proof}

\end{document}